\documentclass[a4paper,final]{amsart}

\usepackage[utf8]{inputenc}
\usepackage[ngerman, english]{babel}

\usepackage{amsthm}
\usepackage{amsmath}
\usepackage{amssymb}
\usepackage{amsfonts}
\usepackage{thmtools}
\usepackage{ytableau}
\usepackage{paralist} 
	\setdefaultenum{(1)}{(i)}{}{}
\usepackage{url}
\usepackage[numbers,sort]{natbib}

\usepackage{xcolor}
\usepackage{graphicx}
\usepackage{tikz}

\usepackage[colorlinks=true, linkcolor=black, citecolor=black, urlcolor=blue]{hyperref}

\addto\extrasenglish{%
}
\addto\extrasenglish{%
}
\addto\extrasenglish{%
}

\newtheorem{thm}{Theorem}[section]
\newtheorem{cor}[thm]{Corollary}
\newtheorem{prop}[thm]{Proposition}
\newtheorem{lem}[thm]{Lemma}

\newtheorem*{cla*}{Claim}
\newtheorem{defi}[thm]{Definition}

\theoremstyle{definition}

\newtheorem*{rem*}{Remark}
\newtheorem{exa}[thm]{Example}
\newtheorem*{ack}{Acknowledgment}

\newtheoremstyle{efronremark}{6pt}{6pt}{}{}{\itshape}{\quad}{ }{\thmnote{#3}}
\theoremstyle{efronremark}

\numberwithin{equation}{section}
\ytableausetup{centertableaux, boxsize=1.2em}

\newcommand{\bs}[1]{\boldsymbol{#1} }
\newcommand{\mc}[1]{\mathcal{#1} }

\newcommand{\N}{\mathbb N}

\newcommand{\Z}{\mathbb Z}
\newcommand{\C}{\mathbb C}

\newcommand{\Sym}{\operatorname{\mathit{Sym}}}
\newcommand{\QSym}{\operatorname{\mathit{QSym}}}

\newcommand{\set}[1]{\left\{ #1 \right\}}

\newcommand{\id}{\operatorname{id}}

\newcommand{\Ch}{\operatorname{\mathit{Ch}}}

\newcommand{\sh}{\operatorname{sh}}
\newcommand{\osh}{\operatorname{osh}}
\newcommand{\ish}{\operatorname{ish}}

\newcommand{\D}{\operatorname{\mathit{D}}}
\newcommand{\DC}{\operatorname{\mathit{D^c}}}
\newcommand{\AD}{\operatorname{\mathit{AD}}}
\newcommand{\nAD}{\operatorname{\mathit{nAD}}}
\newcommand{\ND}{\operatorname{\mathit{ND^c}}}

\newcommand{\SCT}{\operatorname{SCT}}

\newcommand{\fS}{\mathfrak{S}}
\newcommand{\fs}{s}
\newcommand{\myif}{\quad \text {if }}

\newcommand{\spa}{\operatorname{span}}
\newcommand{\col}{\operatorname{col}}
\newcommand{\He}{H_n(0)}

\newcommand{\skc}{/ \! \! / }
\newcommand{\ab}{{\alpha  / \! \! / \beta}}
\newcommand{\parts}[2]{{(#1_1, \dots, #1_{#2})}}

\newcommand{\pleq}{\preceq}
\newcommand{\dleq}{\trianglelefteq}
\newcommand{\dlneq}{\vartriangleleft}

\newcommand{\supp}{\operatorname{supp}}
\newcommand{\End}{\operatorname{End}}

\newcommand{\att}{\rightsquigarrow}
\newcommand{\natt}{\not \rightsquigarrow}

\newcommand{\cell}{\square}

\newcommand{\sa}{\bs S_{\alpha}}
\newcommand{\sab}{\bs S_{\ab}}
\newcommand{\se}{\bs S_{\alpha,E}} 
\newcommand{\sse}{\bs S_{\ab,E}}

\title[Decomposition of 0-Hecke modules]{The decomposition of 0-Hecke modules associated to quasisymmetric Schur functions} 
\author{Sebastian König}
\address{
	Leibniz Universität Hannover \\	
	Institute of Algebra, Number Theory and Discrete Mathematics \\
	Welfengarten 1 \\
	30167 Hannover \\
	Germany
}
\email{sebastian.koenig@math.uni-hannover.de}

\hypersetup{%
	pdfauthor = {Sebastian König},
	pdftitle = {The decomposition of 0-Hecke modules associated to quasisymmetric Schur functions},
	pdfkeywords = {0-Hecke algebra, composition tableau, quasisymmetric function, Schur function}
}

\begin{document}


%

\begin{abstract} 	
	Recently Tewari and van Willigenburg constructed modules of the 0-Hecke algebra that are mapped to the quasisymmetric Schur functions by the quasisymmetric characteristic and decomposed them into a direct sum of certain submodules. We show that these submodules are indecomposable by determining their endomorphism rings.	 
\end{abstract}
	
\maketitle

\section{Introduction}

Since the 19th century mathematicians have been interested in the Schur fuctions $s_\lambda$ and their various properties. For example, they form an orthonormal basis of $\Sym$, the Hopf algebra of symmetric functions and are the images of the irreducible characters of the symmetric groups under the characteristic map \cite{Stanley1999}. 
The symmetric functions are contained in the Hopf algebra $\QSym$ of quasisymmetric functions defined in 1984 \cite{Gessel1984}. 
An introduction to $\QSym$ can be found in \cite{Grinberg2014}.

There is a representation theoretic interpretation of $\QSym$ as well.
The 0-Hecke algebra $\He$ is a deformation of the group algebra $\C\fS_n$ of the symmetric group obtained by replacing the generators $(i,i+1)$ of $\fS_n$ by projections $\pi_i$ satisfying the same braid relations.
Let $\mc G_0\left( \He \right)$ denote the Grothendieck group of the finitely generated $H_n(0)$-modules and $\mc G:=\bigoplus_{n \geq 0}\mc G_0\left( \He \right)$. 
The connection to $\QSym$ was given in \cite{Duchamp1996} by defining an algebra isomorphism $\Ch\colon \mc G \to \QSym$ called quasisymmetric characteristic. 

As $\Sym$ is contained in $\QSym$, one may ask whether there are quasisymmetric analogues of the Schur functions.
One proposal are the quasisymmetric Schur functions  $\mc S_\alpha$ \cite{Haglund2011}. They form a basis of $\QSym$ and nicely refine the Schur functions via
\begin{align*}
s_\lambda = \sum_{\widetilde \alpha = \lambda} \mc S_\alpha
\end{align*}
where $\lambda$ is a partition and the sum runs over all compositions $\alpha$ that rearrange $\lambda$  \cite{Haglund2011} (see \autoref{sec:compositions_and_composition_tableaux} for definitions).
In \cite{Bessenrodt2011} skew quasisymmetric Schur functions $\mc S_\ab$ were defined and a Littlewood-Richardson rule for expressing them in the basis of quasisymmetric Schur functions was proved.

Another basis of $\QSym$ sharing properties with the Schur functions is given by the dual immaculate functions \cite{Berg2014}. Indecomposable 0-Hecke modules whose images under $\Ch$ are the dual immaculate functions were defined in \cite{Berg2015}.

In \cite{Tewari2015} Tewari and van Willigenburg constructed modules $\sa$ of the 0-Hecke algebra that are mapped to $\mc S_\alpha$ by $\Ch$. Each $\sa$ has a $\C$-basis parametrized  by a set of tableaux. By using an equivalence relation, they divided this set into equivalence classes, obtained a submodule $\se$ of $\sa$ for each such equivalence class $E$ and decomposed $\sa$ as $\sa = \bigoplus_E \se$.
In the same way they defined and decomposed skew modules $\sab$ whose image under $\Ch$ is $\mc S_\ab$.

This article is mainly concerned with the modules $\sa$ and $\se$. 
In \cite{Tewari2015}, for a special equivalence class $E_\alpha$ it was shown
that $\bs S_{\alpha, E_\alpha}$ is indecomposable. Yet, the question of the indecomposability of the $\se$ in general remained open.
The goal of this paper is to answer this question. We show that for each $\se$ the ring of $\He$-endomorphisms is $\C\id$ so that $\se$ is indecomposable. As a consequence, $\sa = \bigoplus_E \se$ is a decomposition into indecomposable submodules.

The structure of the paper is as follows. In \autoref{sec:preliminaries} we introduce the combinatorial and algebraic background and then review the modules $\sab$ and $\sse$. \autoref{sec:0-hecke_operation_and_chains} is devoted to a related $\He$-operation on chains of a composition poset.
From this we obtain an argument crucial for proving the main results in \autoref{sec:endomorphism_ring_of_0-Hecke_modules}.

\section{Background}
\label{sec:preliminaries}

We set $\N := \set{1,2,\dots}$ and always assume that $n\in \N$. For $a,b\in \Z$ we define the \emph{discrete interval} $[a,b]:=\set{c\in \Z\mid a\leq c \leq b}$ and may use the shorthand $[a] := [1,a]$. For a set $X$, $\spa_\C X$ is the $\C$-vector space with basis $X$.

\subsection{Symmetric groups and 0-Hecke algebras}

The \emph{symmetric group} $\fS_n$ is the group of all permutations of the set $[n]$. We proceed by reviewing $\fS_n$ as Coxeter group. More details can be found in \cite{Bjorner2006}.

As a Coxeter group $\fS_n$ is generated by the adjacent transpositions $\fs_i := (i,i+1)$ for  $i= 1,\dots, n-1$ which satisfy
\begin{alignat*}{3}
\fs_i^2 &= 1,   			\\
\fs_i\fs_{i+1}\fs_i &= \fs_{i+1}\fs_i\fs_{i+1}, 	\\
\fs_i\fs_j &= \fs_j \fs_i 		\ \text{if } |i-j|\geq 2. 
\end{alignat*}
The latter two relations are called \emph{braid relations}.
Let $\sigma \in \fS_n$. We can write $\sigma$ as product $\sigma = s_{j_k} \cdots s_{j_1}$.
If $k$ is minimal among such expressions, $s_{j_k} \cdots s_{j_1}$ is a \emph{reduced word} for $\sigma$ and $l(\sigma) := k$ is the \emph{length} of $\sigma$. 

The \emph{support} of $\sigma$ is $\supp(\sigma) =\{ i\in [n-1] \mid \text{$\fs_i$ appears in a reduced word of $\sigma$}\}$.
One assertion of the \emph{word property} of $\fS_n$ \cite[Theorem 3.3.1]{Bjorner2006} is that a reduced word of $\sigma$ can be transformed into any other reduced word of $\sigma$ by applying a sequence of braid relations. 
Thus, for each reduced word of $\sigma$ the set of indices occurring in it is $\supp(\sigma)$.

Let $\sigma,\tau \in \fS_n$. The \emph{left weak Bruhat order} $\leq_L$ is the partial order on $\fS_n$ given by
\begin{align*}
\sigma \leq_L \tau \iff 
\begin{aligned}
&\tau = \fs_{i_k} \cdots \fs_{i_1} \sigma, \\
&l(\fs_{i_r} \cdots \fs_{i_1} \sigma) = l(\sigma) + r \text{ for } r=1,\dots ,k.
\end{aligned}
\end{align*}
In the sequel we often drop the adjective \emph{weak}. The following \autoref{thm:properties_bruhat_order} gathers immediate consequences of the definition.

\begin{prop} 
		\label{thm:properties_bruhat_order}
	Let $\sigma,\tau \in \fS_n$.
	\begin{compactenum}
		\item $\sigma \leq_L \tau \iff l(\tau\sigma^{-1}) = l(\tau) - l(\sigma)$.
		\item  If $\sigma\leq_L \tau$ then the reduced words for $\tau\sigma^{-1}$ are in bijection with saturated chains in the left Bruhat poset from $\sigma$ to $\tau$ via 
		\begin{align*}
		\fs_{i_k} \cdots \fs_{i_1} \quad \longleftrightarrow \quad \sigma \lessdot_L \fs_{i_1} \sigma \lessdot_L \fs_{i_2}\fs_{i_1} \sigma \lessdot_L\dots \lessdot_L \fs_{i_k} \cdots \fs_{i_1}\sigma = \tau.
		\end{align*}
		\item The left Bruhat poset $(\fS_n, \leq_L)$ is graded by the length function.
	\end{compactenum}
\end{prop}

\begin{thm}[{\cite[Corollary 3.2.2]{Bjorner2006}}]
	\label{thm:bruhat_interval}
	Let $\sigma,\tau\in \fS_n$. The interval in left Bruhat order $[\sigma, \tau] := \left\{\rho \in \fS_n \mid \sigma \leq_L \rho \leq_L  \tau \right\}$  is a graded lattice with rank function $\rho \mapsto l(\rho \sigma^{-1})$.
\end{thm}
 
Next, we define the 0-Hecke algebra $\He$. We use the presentation as in \cite{Tewari2015} and refer  to \cite[Ch. 1]{Mathas1999} for details.

\begin{defi}
	\label{thm:0-Hecke algebra}
	The \emph{0-Hecke algebra} $\He$ is the unital associative $\C$-algebra generated by the elements $\pi_1,\pi_2,\dots ,\pi_{n-1}$ subject to relations
	\begin{align*}
	\pi_i^2 &= \pi_i, \\
	\pi_i\pi_{i+1}\pi_i &= \pi_{i+1} \pi_{i} \pi_{i+1}, \\
	\pi_i \pi_j &= \pi_j \pi_i  \text{ if } |i-j| \geq 2. 	 
	\end{align*}
\end{defi}

Note that the $\pi_i$ are projections satisfying the same braid relations as the $\fs_i$. 
Let $\sigma \in \fS_n$. We define $\pi_\sigma := \pi_{j_k} \cdots \pi_{j_1}$ where $\fs_{j_k} \cdots \fs_{j_1}$ is a reduced word for $\sigma$. The word property ensures that this is well defined. Multiplication is given by
\begin{align*}
\pi_i \pi_\sigma = 
\begin{cases}
\pi_{\fs_i \sigma} & \text{if}\ l(\fs_i \sigma ) > l(\sigma) \\
\pi_{\sigma}       & \text{if}\ l(\fs_i \sigma ) < l(\sigma) \\
\end{cases}
\end{align*}
for $i = 1,\dots, n-1$. 
As a consequence, $\set{\pi_{\tau} \mid \tau \in \fS_n}$ spans $\He$ over $\C$. One can  also  show that it is a $\C$-basis of $\He$.

\subsection{Compositions and composition tableaux}
\label{sec:compositions_and_composition_tableaux}

A \emph{composition} $\alpha = \parts{\alpha}l$ is a finite sequence of positive integers. 
The \emph{length} and the \emph{size} of $\alpha$ are given by $l(\alpha):= l$ and $|\alpha| := \sum_{i=1}^l \alpha_i$, respectively.
The $\alpha_i$ are called \emph{parts} of $\alpha$. If $\alpha$ has size $n$, $\alpha$ is called \emph{composition of $n$} and we write $\alpha \vDash n$.
A \emph{partition} is a composition whose parts are weakly decreasing. 
We write $\lambda \vdash n$ if $\lambda$ is a partition of size $n$.
For a composition $\alpha$ we denote the partition obtained by sorting the parts of $\alpha$ in decreasing order by $\widetilde \alpha$. The \emph{empty composition} $\emptyset$ is the unique composition of length and size $0$.

\begin{exa} 
	For $\alpha=(1,4,3)\vDash 8$ we have  $\widetilde \alpha = (4,3,1)\vdash 8$.
\end{exa}

A \emph{cell} $(i,j)$ is an element of $\N \times \N$.
A finite set of cells is called \emph{diagram}. Diagrams are visualized in English notation. That is, for each cell $(i,j)$ of a diagram we draw a box at position $(i,j)$ in matrix coordinates.
The \emph{diagram} of $\alpha\vDash n$ is the set $\set{(i,j)\in \N \times \N \mid  i \leq l(\alpha), j \leq \alpha_i}$. So, we display the diagram of $\alpha$ by putting $\alpha_i$ boxes in row $i$ where the top row has index~$1$.
We may identify $\alpha$ with its diagram.

\begin{exa}
	\begin{align*}
		(1,4,3) = 
		\begin{ytableau} 
		 \\
			&&& \\
			&&\\
		\end{ytableau}
	\end{align*}
\end{exa}

Next, we will introduce standard composition tableaux and a related poset of compositions which arose in \cite{Bessenrodt2011}.

\begin{defi}
	The \emph{composition poset} $\mc L_c$ is the set of all compositions together with the partial order $\leq_c$ given as the transitive closure of the following covering relation. For compositions $\alpha$ and  $\beta = ( \beta_1, \dots,  \beta_l)$
	\begin{align*}
		 \beta \lessdot_c  \alpha  \iff  \begin{aligned}
		 \alpha &= (1, \beta_1, \dots,  \beta_l) \text{ or}\\
		 \alpha &= ( \beta_1, \dots,   \beta_k +1 , \dots,  \beta_l) \text{ and $ \beta_i \neq  \beta_k$ for all $i<k$.}
		\end{aligned} 
	\end{align*}	
\end{defi}  

In other words, $\beta$ is covered by $\alpha$ in $\mc L_c$ if and only if the diagram of $\alpha$ can be obtained from the diagram of $\beta$ by adding a box as new first row or appending a box to a row which is the topmost row of its length in $\beta$.

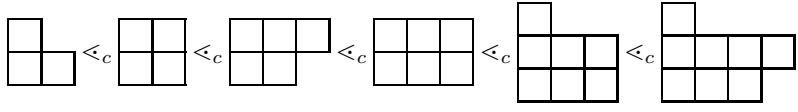
\begin{exa}
	\label{exa:chain_in_composition_poset}
	\[	
		\ydiagram{1,2} \lessdot_c 
		\ydiagram{2,2} \lessdot_c  
		\ydiagram{3,2} \lessdot_c 
		\ydiagram{3,3} \lessdot_c 
		\ydiagram{1,3,3} \lessdot_c
		\ydiagram{1,4,3} 
	\]
\end{exa}

Let $\alpha$ and $\beta$ be two compositions such that $\beta \leq_c \alpha$. In this situation we always assume that the diagram of $\beta$ is moved to the bottom of the diagram of $\alpha$, and we define the \emph{skew composition diagram}  (or \emph{skew shape}) $\ab$ to consist of all cells of $\alpha$ which are not contained in $\beta$.
Moreover, we define $\osh(\ab)= \alpha$ and $\ish(\ab)= \beta$ as the \emph{outer} and the \emph{inner shape} of $\ab$, respectively. 
The size of a skew shape is $|\ab| := |\alpha| - |\beta|$.
If $\beta = \emptyset$ then $\ab=\alpha$ is an ordinary composition diagram and we call $\ab$ \emph{straight}.

\begin{exa} In the following the cells of the inner shape are gray.
	\[
		(1,4,3) \skc (1,2) =
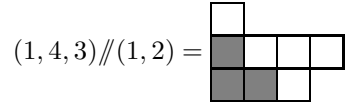
		\begin{ytableau} 
			 \\
			*(gray)&		&& \\
			*(gray)&*(gray)	&\\
		\end{ytableau}
	\]	
\end{exa}

Let $D$ be a diagram. A \emph{tableau} $T$ of shape $D$ is a map $T\colon D \to \N$. It is visualized by filling each $(i,j)\in D$ with $T(i,j)$.

\begin{defi}
	Let $\ab$ be a skew shape of size $n$. A \emph{standard composition tableau} (SCT) of shape $\ab$ is a bijective filling $T\colon \ab \to [n]$ satisfying the following conditions:
	\begin{compactenum}
		\item
		The entries are decreasing in each row from left to right.
		\item
		The entries are increasing in the first column from top to bottom.		
		\item
		(Triple rule). Set $T(i,j):= \infty$ for all $(i,j)\in \beta$.
		If $(j,k)\in \alpha \skc \beta$ and $(i,k-1)\in \alpha$ such that $j > i$ and $T(j,k) < T(i,k-1)$ then $(i,k) \in \alpha$ and $T(j,k) < T(i,k)$.

	\end{compactenum}
\end{defi}

The set of composition tableaux of shape $\ab$ is denoted with $\SCT(\ab)$. 
For an SCT $T$ we write $\sh(T)$ for its shape. 
The notions of outer and inner shape are carried over from $\sh(T)$ to $T$.
We call $T$ \emph{straight} if its shape is straight.

\begin{exa}
	\label{exa:sct}
	 A SCT is shown below.
	\begin{align*}
	T = 
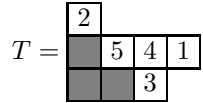
	\begin{ytableau} 
		2 \\ 
		*(gray) & 5 & 4 & 1 \\ 
		*(gray) & *(gray) &  3 \\ 
		\end{ytableau}
	\end{align*}
	We have $\osh(T)= (1,4,3)$ and 	$\ish(T) = (1,2)$.
\end{exa}

Standard composition tableaux encode saturated chains of $\mathcal{L}_c$ in the following way.

\begin{prop}[see {\cite[Proposition 2.11]{Bessenrodt2011}}]
	\label{thm:sct_and_chains}
	Let $\ab$ be a skew composition of size $n$. For $T\in \SCT(\ab)$, 
	\begin{align*}
		\beta = \alpha^n  \lessdot_c \alpha^{n-1}  \lessdot_c \dots  \lessdot_c \alpha^0 =\alpha
	\end{align*}
	given by
	\begin{align}  \label{eq:tableau_chain}
	\alpha^n = \beta, \quad \alpha^{k-1} = \alpha^k \cup T^{-1}(k) \quad \text{for} \quad k= 1,\dots, n
	\end{align}
	is a saturated chain in $\mathcal{L}_c$. Moreover, we obtain a bijection from $\SCT(\ab)$ to the set of saturated chains in $\mathcal{L}_c$ from $\beta$ to $\alpha$ by mapping each tableau of $\SCT(\ab)$ to its corresponding chain given by \eqref{eq:tableau_chain}. 	 
\end{prop}

\begin{exa}
	The SCT from \autoref{exa:sct} 	corresponds to the chain from \autoref{exa:chain_in_composition_poset}.	
\end{exa}

From the perspective of \autoref{thm:sct_and_chains}, the triple rule reflects the fact that by adding a cell to a row of a composition diagram, a covering relation in $\mc L_c$ is established if and only if the row in question is the topmost row of its length.

Some of the upcoming notions already played a role in \cite{Tewari2015}.
Let $(i,j)$ and $(i',j')$ be two cells.
Define $r(i,j) := i$ and $c(i,j) := j$ the \emph{row} and the \emph{column} of $(i,j)$, respectively. 
We say that $(i,j)$ \emph{attacks} $(i',j')$  and write $(i,j)\att (i',j')$ if  $j=j'$ and $i\neq i'$ or $j = j'-1$ and $i<i'$.
That is, the two cells are distinct and appear either in the same column or in adjacent columns and $(i',j')$ is located strictly below and right of $(i,j)$.
 
Let $T$ be an SCT and $i,j\in T$ two entries.
We refer to the \emph{row} and the \emph{column} of $i$ in $T$ by $r_T(i) := r(T^{-1}(i))$ and $c_T(i):= c(T^{-1}(i))$, respectively. 
We say that $i$ \emph{attacks} $j$ in $T$ and write $i\att_T j$ if $T^{-1}(i) \att T^{-1}(j)$. The index $T$ may be omitted if it is clear from context.
Note that $i\att_T j$ implies $i\neq j$. 

For two sets of cells $C_1, C_2 \subseteq \N^2$ we say $C_1$ attacks $C_2$ and write $C_1 \att C_2$ if there are cells $c_1\in C_1$ and $c_2 \in C_2$ such that $c_1 \att c_2$. 
If $c(c_1)\leq c(c_2)$ for all $c_1\in C_1, c_2\in c_2$, $C_1$ is called \emph{left} of $C_2$. If $c(c_1)< c(c_2)$ for all $c_1\in C_1, c_2\in c_2$, $C_1$ is \emph{strictly left} of $C_2$. In the same way we use these notions for sets of entries of an SCT.

\begin{exa}
	In the tableau from \autoref{exa:sct} we have $2\natt 3, 2\att \set{3,5}, 3\att 4$ and $3$ is left of $\set{1,4}$.
\end{exa}

\begin{defi}
	\label{thm:descents}
	 Let $T$ be an SCT of size $n$.
	\begin{compactenum}
		\item $\D(T) = \{i\in [n-1] \mid c(i) \leq c(i+1)\}$ is the \emph{descent set} of~$T$.
		\item $\AD(T) = \{i\in \D(T) \mid i\att i+1 \}$ is the \emph{set of attacking descents} of~$T$.
		\item $\nAD(T) = \{i\in \D(T) \mid i \notin \AD(T) \}$ is the \emph{set of non-attacking descents} of~$T$.
		\item[(1')] $\DC(T) = \{i\in [n-1] \mid c(i+1) < c(i)\} = [n-1]\setminus \D(T)$ is the \emph{ascent set} of $T$.
		\item[(2')] $\ND(T) = \{i\in \DC(T) \mid  i+1$ left neighbor of $i\}$ is the \emph{set of neighborly ascents} of $T$.
	\end{compactenum}
\end{defi}

\begin{exa}
Let $T$ be the tableau from \autoref{exa:sct}. Then
	$\D(T) = \{2,3\}$, 
	$\AD(T) = \{3\}$, 
	$\DC(T) =\{1,4\}$ and
	$\ND(T)  = \{4\}$. 
\end{exa}

\subsection{0-Hecke modules of standard composition tableaux}

In this subsection we introduce the skew 0-Hecke modules $\sab$ and $\sse$ and review related results from \cite{Tewari2015}. This includes the special cases $\sa$ and $\se$.

\begin{thm}[{\cite[Theorem 9.8]{Tewari2015}}]
	\label{thm:sct_modules}
	Let $\ab$ be a skew composition of size $n$. 
	Then $\sab:=\spa_\C \SCT(\ab)$ is a $\He$-module with respect to the following action. For $T\in \SCT(\ab)$ and $i=1,\dots, n-1$,
	\begin{align*}
	\pi_i T = 
	\begin{cases}
	T &\myif i \notin \D(T) \\
	0 &\myif i \in \AD(T) \\ 
	s_iT  &\myif i \in \nAD(T) 
	\end{cases}
	\end{align*}
	where $s_iT$ is the tableau obtained from $T$ by interchanging $i$ and $i+1$.
\end{thm}

The module $\sa$ is called \emph{straight} if $\alpha= \ab$ is a composition.
Even though the main results of this paper only concern straight modules, we introduce the more general concept of skew modules here as they naturally arise in the context of the 0-Hecke action on chains of $\mc L_c$ in  \autoref{sec:0-hecke_operation_and_chains}.

\begin{exa}
	\label{exa:0-hecke_action}
	Consider the SCT 
	$T=
	\begin{ytableau} 
		1 \\ 
		6 & 5 & 4 & 3 \\ 
		8 & 7 &  2 \\ 
	\end{ytableau}.$ 
	Then $\D(T) = \set{1,2,6}$,	
	\[
		\pi_i T = 
		\begin{cases}
			T &\text{for $i=3,4,5,7$} \\
			0 &\text{for $i=6$} \\
			s_iT& \text{for $i=1,2$},
		\end{cases}
		\quad
		s_1T = 
		\begin{ytableau} 
			\bs2 \\ 
			6 & 5 & 4 & 3 \\ 
			8 & 7 & \bs1
		\end{ytableau}
		\quad \text{and} \quad
		s_2T =
		\begin{ytableau} 
			1 \\ 
			6 & 5 & 4 & \bs 2 \\ 
			8 & 7 & \bs3
		\end{ytableau}.
	\]
\end{exa}

The following relation gives rise to a decomposition of $\sab$.
Let $\ab$ be a skew composition of size $n$ and $T_1,T_2 \in \SCT(\ab)$. An equivalence relation $\sim$ on $\SCT(\ab)$ is given by
\begin{align*}
	T_1 \sim T_2 \iff      
	\text{in each column the relative orders of entries in $T_1$ and $T_2$ coincide.}
\end{align*}
 For example, the tableaux shown in \autoref{fig:example_for_E} form an equivalence class under $\sim$.
 We denote the \emph{set of equivalence classes} under $\sim$ on $\SCT(\ab)$ by $\mc E(\ab)$. 
 
For $E\in \mc E(\ab)$ define $\sse = \spa_\C E$. 
Observe that the definition of the 0-Hecke action on SCTx in \autoref{thm:sct_modules} implies that $\sse$ is a $\He$-submodule of $\sab$. Thus we have the following.
\begin{prop}[{\cite[Lemma 6.6]{Tewari2015}}]
	\label{thm:decomposition_in_equivalence_classes}
	Let $\ab$ be a skew composition. Then we have
	$\sab = \bigoplus_{E\in \mc E(\ab)} \sse$ as $\He$-modules.
\end{prop}

The main result of this paper is that the $\He$-endomorphism ring of each straight module $\se$ is $\C \id$ and, therefore, we obtain a decomposition of $\sa$ into indecomposable submodules from  \autoref{thm:decomposition_in_equivalence_classes}. We continue by studying the $\sse$ and their equivalence classes more deeply. 

Let $\ab$ be a skew composition of size $n$, $E\in \mc E(\ab)$ and $T_1,T_2\in E$. In \cite[Section 4]{Tewari2015} it is shown that a partial order $\pleq$ on $E$ is given by
\begin{align*}
	T_1 \pleq T_2 \iff \exists \sigma \in \fS_n \text{ such that } \pi_\sigma T_1 = T_2.
\end{align*}

 We refer to the poset $(E,\pleq)$ simply by $E$. An example is shown in \autoref{fig:example_for_E}. The following theorem summarizes results of \cite[Section 6]{Tewari2015}.
 
\begin{figure}
	\scalebox{1}{%
		\begin{tikzpicture} 
		\newcommand*{\xdist}{*3}
		\newcommand*{\ydist}{*2.2}
		
		\node (n0) at (0.00\xdist,0\ydist) 
		{
			$T_{0,E} =
			\begin{ytableau} 
			1 \\ 
			6 & 5 & 4 & 3 \\ 
			8 & 7 & 2 \\ 
			\end{ytableau}  
			$ 
		}; 
		\node (n1) at (-0.50\xdist,1\ydist) 
		{
			$\begin{ytableau} 
			2 \\ 
			6 & 5 & 4 & 3 \\ 
			8 & 7 & 1 \\ 
			\end{ytableau}  
			$ 
		}; 
		\node (n2) at (0.50\xdist,1\ydist) 
		{
			$\begin{ytableau} 
			1 \\ 
			6 & 5 & 4 & 2 \\ 
			8 & 7 & 3 \\ 
			\end{ytableau}  
			$ 
		}; 
		\node (n3) at (-0.50\xdist,2\ydist) 
		{
			$\begin{ytableau} 
			3 \\ 
			6 & 5 & 4 & 2 \\ 
			8 & 7 & 1 \\ 
			\end{ytableau}  
			$ 
		}; 
		\node (n4) at (0.50\xdist,2\ydist) 
		{
			$\begin{ytableau} 
			2 \\ 
			6 & 5 & 4 & 1 \\ 
			8 & 7 & 3 \\ 
			\end{ytableau}  
			$ 
		}; 
		\node (n5) at (-0.50\xdist,3\ydist) 
		{
			$\begin{ytableau} 
			4 \\ 
			6 & 5 & 3 & 2 \\ 
			8 & 7 & 1 \\ 
			\end{ytableau}  
			$ 
		}; 
		\node (n6) at (0.50\xdist,3\ydist) 
		{
			$\begin{ytableau} 
			3 \\ 
			6 & 5 & 4 & 1 \\ 
			8 & 7 & 2 \\ 
			\end{ytableau}  
			$ 
		}; 
		\node (n7) at (0.00\xdist,4\ydist) 
		{
			$T_{1,E} =
			\begin{ytableau} 
			4 \\ 
			6 & 5 & 3 & 1 \\ 
			8 & 7 & 2 \\ 
			\end{ytableau}  
			$ 
		}; 
		
		\draw [thick, ->] (n0) -- (n1) node [near start, left] {$\pi_{1}$}; 
		\draw [thick, ->] (n0) -- (n2) node [near start, left] {$\pi_{2}$}; 
		\draw [thick, ->] (n1) -- (n3) node [near start, left] {$\pi_{2}$}; 
		\draw [thick, ->] (n2) -- (n4) node [near start, left] {$\pi_{1}$}; 
		\draw [thick, ->] (n3) -- (n5) node [near start, left] {$\pi_{3}$}; 
		\draw [thick, ->] (n3) -- (n6) node [near start, left] {$\pi_{1}$}; 
		\draw [thick, ->] (n4) -- (n6) node [near start, left] {$\pi_{2}$}; 
		\draw [thick, ->] (n5) -- (n7) node [near start, left] {$\pi_{1}$}; 
		\draw [thick, ->] (n6) -- (n7) node [near start, left] {$\pi_{3}$}; 			
		\end{tikzpicture}				
	}
	\caption{A poset $(E,\pleq)$. Each covering relation is labeled with the 0-Hecke generator $\pi_i$ realizing it. 
	}
	\label{fig:example_for_E}
\end{figure}
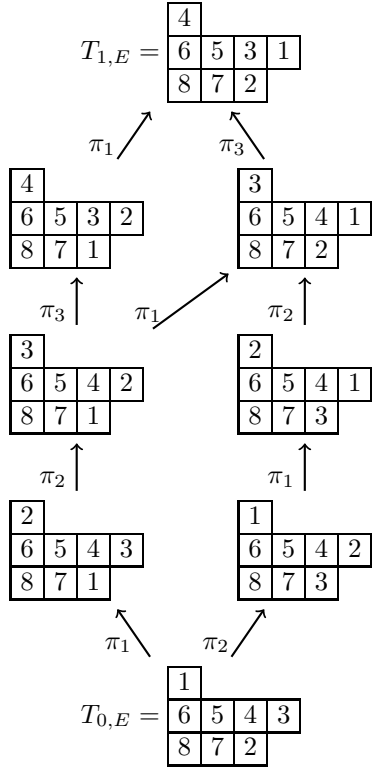

\begin{samepage}
\begin{thm}
	\label{thm:source_and_sink_tableau}
	 Let $\ab$ be a skew composition, $E\in \mc E(\ab)$ and $T\in E$.
	\begin{compactenum}
		\item $T$ is minimal according to $\pleq$ if and only if $\DC(T)= \ND(T)$. There is a unique tableau $T_{0,E}\in E$ which satisfies these conditions called \emph{source tableau} of $E$.
		\item $T$ is maximal according to $\pleq$ if and only if $\D(T)= \AD(T)$. There is a unique tableau $T_{1,E}\in E$ which satisfies these conditions called \emph{sink tableau} of $E$.
	\end{compactenum}
	In particular, $\sse$ is a cyclic module generated by $T_{0,E}$.
\end{thm}
\end{samepage}

A source and a sink tableau can be observed in \autoref{fig:example_for_E}. Next, we establish a connection between $E$ and an interval of the left Bruhat order. To do this we introduce the notion of \emph{column words}.
 Given $T\in \SCT(\ab)$ and $j\geq 1$, let $w_j$ be the word obtained by reading the entries in the $j$th column of $T$ from top to bottom. Then $\col_T= w_1w_2 \cdots$ is the \emph{column word} of $T$. Clearly, $\col_T$ can be regarded as an element of $\fS_n$ (in one-line notation).
 
\begin{exa}
	The tableau $T_{0,E}$ from \autoref{fig:example_for_E} has $\col_{T_{0,E}} = 16857423\in \fS_8$. 
\end{exa}

\begin{lem}[{\cite[Lemma 4.4]{Tewari2015}}]
	\label{thm:tableaux_cover_impies_column_words_cover}
	Let $T_1$ be an SCT, $i\in \nAD(T_1)$ and $T_2 = \pi_i T_1$. Then $\col_{T_2} = \fs_i \col_{T_1}$ and $l(\col_{T_2}) = l(\col_{T_1}) + 1$. That is, $\col_{T_2}$ covers $\col_{T_1}$ in left Bruhat order.
\end{lem}

The following statement is similar to \cite[Lemma 4.3]{Tewari2015}.
\begin{lem}
	\label{thm:column_word_and_0-hecke_operation}
	Let $T_1$ and $T_2$ be two SCTx such that   $\pi_{i_p} \cdots \pi_{i_1} T_1 =  T_2$. Then there is a subsequence $j_q, \dots, j_1$ of $i_p, \dots,i_1$ such that
	\begin{compactenum}
		\item
		$T_2 = \pi_{j_q} \cdots \pi_{j_1}T_1$,
		\item 
		$\fs_{j_q} \cdots \fs_{j_1}$ is a reduced word for $\col_{T_2}\col^{-1}_{T_1}$.
	\end{compactenum}
	In particular, $T_2 = \pi_{\col_{T_2}\col^{-1}_{T_1}} T_1$.
\end{lem}

\begin{proof} From the definition of the 0-Hecke operation follows that we can find a subsequence $j_q, \dots, j_1$ of $i_p, \dots,i_1$ of minimal length such that $T_2 = \pi_{j_q} \cdots \pi_{j_1}T_1$. If $q=0$ then $T_2 = T_1$ and the result is trivial. If $q=1$ set $i :=j_1$. Then by the minimality of $q$, $T_2 \neq T_1$ and thus $i\in \nAD(T_1)$. Now \autoref{thm:tableaux_cover_impies_column_words_cover} shows that $\fs_i$ is a reduced word for $\col_{T_2}\col^{-1}_{T_1}$. If $q>1$ use the case $q=1$ iteratively.
\end{proof}

\begin{thm}[{\cite[Theorem 6.18]{Tewari2015}}]
	\label{thm:E_isomorphic_to_bruhat_interval}
	Let $\ab$ be a skew composition, $E\in \mc E(\ab)$ and $I = [\col_{T_{0,E}}, \col_{T_{1,E}}]$ an interval in left Bruhat order. Then the map 
	$\col \colon E\to I$, $T\mapsto\col_T$ is a poset isomorphism. In particular, $E$ is a graded lattice with rank function $\delta\colon T\mapsto l(\col_T \col^{-1}_{T_{0,E}})$.
\end{thm}

Actually, \autoref{thm:source_and_sink_tableau}, \autoref{thm:tableaux_cover_impies_column_words_cover} and \autoref{thm:column_word_and_0-hecke_operation} are everything needed to prove \autoref{thm:bruhat_interval} as in \cite{Tewari2015}. They imply that $\col$ (and its inverse) map maximal chains to maximal chains.
Note that from \autoref{thm:E_isomorphic_to_bruhat_interval} and \autoref{thm:properties_bruhat_order} follows that for $T_1 \pleq T_2$ saturated chains from $T_1$ to $T_2$ correspond to reduced words for $\col_{T_2}\col^{-1}_{T_1}$.

\begin{cor}
	\label{thm:rank_in_E_and_column_word}
	Let $T_1$ and $T_2$ be two SCTx of size $n$ and $\sigma \in \fS_n$ such that $T_2 = \pi_\sigma T_1$. Then $T_1$ and $T_2$ belong to the same equivalence class under $\sim$. Let $\delta$ be the rank function of that class. Then
	\begin{compactenum}
		\item $\delta(T_2) - \delta(T_1) = l(\col_{T_2}\col^{-1}_{T_1})$,
		\item $\delta(T_2) - \delta(T_1) \leq l(\sigma)$ where we have equality if and only if $\sigma = \col_{T_2}\col_{T_1}^{-1}$.
	\end{compactenum}
\end{cor}

\begin{proof} Since $T_2 = \pi_\sigma T_1$, $T_2 \sim T_1$.
Part (1) follows from the discussion above and (2) is a consequence of (1) and \autoref{thm:column_word_and_0-hecke_operation}.
\end{proof}

We finish this section by preparing another consequence of \autoref{thm:tableaux_cover_impies_column_words_cover} for \autoref{sec:endomorphism_ring_of_0-Hecke_modules}.

\begin{prop}
	\label{thm:mutliple_flips_same_cell}
	
	Let $T$ be an SCT, $i,j\in T$ be such that $i<j$ and $\cell = T^{-1}(i)$.  If in $T$ $i$ is located left of $[i+1,j]$
	and does not attack $[i+1,j]$ then 
	\begin{compactenum}	
		\item  $T':=\pi_{j-1} \cdots \pi_{i+1}\pi_i T\in \SCT,$
		\item $\fs_{j-1} \cdots \fs_{i+1}\fs_i$ is a reduced word for $\col_{T'}\col_T^{-1}$,
		\item $T'(\cell)= j$.
	\end{compactenum}
\end{prop}

\begin{proof}
	Let $T$ be an SCT and $i,j\in T$ such that $i<j$, $i$ is located left of $[i+1,j]$
	and $i \natt[i+1,j]$.  Set $\cell = T^{-1}(i)$.
	We do an induction on $m:= j-i$. If $m=1$ then $i\in \nAD(T)$ and $T' = \pi_i T$. Thus, 	(1) and (3) hold and (2) follows from \autoref{thm:tableaux_cover_impies_column_words_cover}.
	
	Now, let $m>1$.  Since by assumption $i$ is located left of $[i+1,j]$
	and $i~\natt~[i+1,j]$, we can apply the induction hypothesis on $i$ and $j-1$ and obtain that $T'':=\pi_{j-2} \cdots \pi_{i+1}\pi_i T \in \SCT$, $\fs_{j-2} \cdots \fs_{i+1}\fs_i$ is a reduced word for $\col_{T''}\col_T^{-1}$ and $T''(\cell) = j-1$.
	Since the operators $\pi_{j-2}, \dots, \pi_{i+1}, \pi_{i}$ are unable to move $j$, we have $T''^{-1}(j) = T^{-1}(j).$
	By choice of $i$ and $j$,  $\cell \natt T^{-1}(j) = T''^{-1}(j)$ and $\cell$ is left of $T''^{-1}(j)$.
	Thus, $j-1\in \nAD(T'')$ so that $T' = \pi_{j-1}  \pi_{j-2}\cdots \pi_i T=\pi_{j-1}(T'')\in \SCT$ and $T'(\cell) = j$.
	From \autoref{thm:tableaux_cover_impies_column_words_cover} follows that $\col_{T'}\col_T^{-1}=s_{j-1}\col_{T''}\col^{-1}_{T}= \fs_{j-1}\fs_{j-2} \cdots \fs_i$ and that $\fs_{j-1}\fs_{j-2} \cdots \fs_i$ is a reduced word.
\end{proof}

\section{A 0-Hecke action on chains of the composition poset}

\label{sec:0-hecke_operation_and_chains}

In \autoref{thm:sct_and_chains} a bijection between saturated chains in the composition poset $\mc L_c$ and standard composition tableaux is given. 
In this section we study the 0-Hecke action on these chains induced by this bijection. The main result of this section, \autoref{thm:support_of_column_word_and_shape}, will be useful in \autoref{sec:endomorphism_ring_of_0-Hecke_modules}. We begin with some notation.

\begin{defi}
	\label{thm:parts_of_chains_in_L_c}
	Let $T$ be an SCT of shape $\ab$ and size $n$,  $m\in [0,n]$ and $\beta = \alpha^n  \lessdot_c \alpha^{n-1}  \lessdot_c \cdots  \lessdot_c \alpha^0 =\alpha$ the chain in $\mathcal L_c$ corresponding to $T$. 
	 The SCT of shape $\alpha^m \skc \beta$  corresponding to the chain $\alpha^n  \lessdot_c \alpha^{n-1}\lessdot_c \cdots  \lessdot_c \alpha^m$ is denoted by $T^{>m}$.
\end{defi}

\begin{exa}
	For 
	$T=\begin{ytableau}
	1 \\
	*(gray) & *(gray) & 3 \\
	*(gray) & 2
	\end{ytableau}$ 
	we have
	$ 
	T^{>2} =   \begin{ytableau} 
	*(gray) & *(gray) & 1  \\
	*(gray) \end{ytableau}.
	$	
\end{exa}

The following Lemma shows how we can obtain $T^{>m}$ directly from $T$.

\begin{lem}
	\label{thm:cells_and_entries_of_restricted_tableaux}
	Let $T$ be an SCT of size $n$ and shape $\ab$, $\beta = \alpha^n  \lessdot_c \alpha^{n-1}  \lessdot_c \cdots  \lessdot_c \alpha^0 = \alpha$ the chain in $\mathcal{L}_c$ corresponding to $T$ and $m\in [0,n]$.
	\begin{compactenum}
		\item	
		$\alpha^m =  \osh(T^{>m})$.
		\item  We obtain $T^{>m}$ from $T$ by removing the cells containing $1,\dots, m$ and subtracting $m$ from the remaining entries.
	\end{compactenum}
\end{lem}

	\begin{proof}
		Part (1) follows directly from \autoref{thm:parts_of_chains_in_L_c}. 	
		By \autoref{thm:sct_and_chains}, we obtain $T^{>m}$ by successively adding cells with entries $n-m,\dots, 1$ to the inner shape $\beta$ at exactly the same positions where we would add $n,\dots, m+1$ to $\beta$ in order to obtain $T$ from its corresponding chain. This implies part (2).
	\end{proof}

With the first part of \autoref{thm:cells_and_entries_of_restricted_tableaux} we can access the compositions within a chain of a given SCT. We use the following preorder to describe how the 0-Hecke action affects these compositions.

\begin{defi}		
	\begin{compactenum}
		\item 	For  $\alpha = \parts{\alpha}{l}\vDash n$ and $j\in \N$ define $|\alpha|_j = \# \left\{ i \in [l] \mid \alpha_i\geq j \right\}$. 
		\item
		 On the set of compositions of size $n$ we define the preorder $\dleq$ by
		\begin{align*}
			\alpha \dleq \beta \iff \sum_{j=1}^k |\beta|_j \leq \sum_{j=1}^k |\alpha|_j \text{ for all $k \geq 1$.}
		\end{align*} 
		 Moreover, set $\alpha \dlneq \beta \iff  \alpha \dleq \beta$ and $\alpha \neq \beta$.
	\end{compactenum}
\end{defi}

Note that $|\alpha|_j$ is the number of cells in the $j$th column of the diagram of $\alpha$.
	Obviously $\dleq$ is reflexive and transitive. It is not antisymmetric since for example $(2,1) \dleq (1,2)$ and $(1,2) \dleq (2,1)$. In general, for $\alpha, \beta\vDash n$ we have
	\begin{align*}
	\alpha \dleq \beta \text{ and } 
	\beta \dleq \alpha \iff |\alpha|_j = |\beta|_j \text{ for all  $j=1,2,\dots$} \iff \widetilde \alpha = \widetilde \beta.
	\end{align*}	
	If we restrict $\dleq$ to partitions, we obtain the well known dominance order appearing, for example, in   \cite{Stanley1999}.

\begin{exa}
\[
\ydiagram{2,2} \dlneq \ydiagram{3,1}
\qquad  \qquad 
\ydiagram{1,2,2}
\dlneq
\begin{ytableau}
\hfil &&\\
&
\end{ytableau}		
\]	
\end{exa}

\begin{lem}
		\label{thm:dominance_and_hecke_covering}
	Let $\ab$ be a skew composition of size $n$ and $T_1,T_2\in \SCT(\ab)$ be such that $T_2 = \pi_i T_1$ for an $i\in \nAD(T_1)$.
	 Then
	\begin{alignat*}{32}
		\osh(T_2^{>i}) &\dlneq  \osh (T_1^{>i}), \\
		\osh(T_2^{>m}) &= \osh (T_1^{>m}) \text{ for }  m \in [0,n],  m\neq i.
	\end{alignat*}
\end{lem}

\begin{proof}
	Assume $T_1,T_2$ and $i$ as in the assertion. Then $T_2$ is the tableau obtained from $T_1$ by swapping the entries $i$ and $i+1$ of $T_1$. Let $m\in[0,n]$.
	 
	If $m\neq i$ then either $\set{i,i+1} \subseteq [1,m]$ or $\set{i,i+1} \cap [1,m]= \emptyset$.
	Therefore, $T_1^{-1}([1,m]) = T_2^{-1} ([1,m])$ and so from the perspective of \autoref{thm:cells_and_entries_of_restricted_tableaux} we remove the same set of cells from $T_1$ to obtain $T_1^{>m}$ as we remove from $T_2$ to obtain $T_2^{>m}$. That is,	$\sh(T^{>m}_1) = \sh(T^{>m}_2)$.
	
	If $m = i$, set $(r_k,c_k) := T_1^{-1}(k)$ for $k= i,i+1$, $\gamma_1 := \osh(T^{>i}_1)$ and $\gamma_2 := \osh(T^{>i}_2)$. 
	We assume that all composition diagrams appearing here are moved to the bottom of $\alpha$.
	Observe that as $T_2 = \fs_i T_1$, one obtains $\sh(T^{>i}_2)$ from $\sh(T^{>i}_1)$ by moving the cell $(r_{i+1}, c_{i+1})$ to the position $(r_i,c_i)$. Since $\ish(T^{>i}_2) = \beta = \ish(T^{>i}_1)$, we obtain $\gamma_2$ from $\gamma_1$ by this movement.
	 From $i\in \nAD(T_1)$ follows $c_i < c_{i+1}$. That is, we obtain $\gamma_2$ from $\gamma_1$ by moving a cell strictly to the left. Then the definition of $\dleq$ implies $\gamma_2 \dlneq \gamma_1$.
\end{proof}

 \begin{exa} 
 	The $\He$-action on tableaux and the corresponding chains of the composition poset is shown below.
 	\begin{align*}
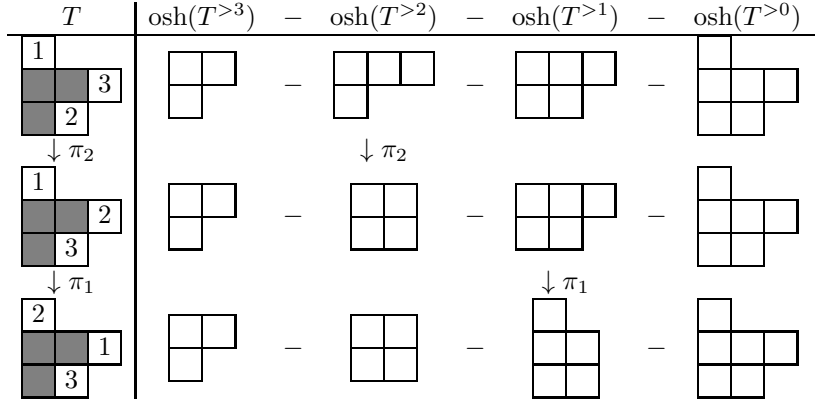
 	\begin{array}{c|ccccccc}
 	T & \osh(T^{>3}) &-& \osh(T^{>2}) &-& \osh(T^{>1})  &-& \osh(T^{>0})\\ \hline
 	\begin{ytableau}
 	1 \\
 	*(gray) & *(gray) & 3 \\
 	*(gray) & 2
 	\end{ytableau}
 	& \ydiagram{2,1} 
 	&-& \ydiagram{3,1} 
 	&-& \ydiagram{3,2} 
 	&-& \ydiagram{1,3,2} \\
 	\downarrow \pi_2 &&& \downarrow \pi_2 
 	\\
 	\begin{ytableau}
 	1 \\
 	*(gray) & *(gray) & 2 \\
 	*(gray) & 3
 	\end{ytableau}
 	& \ydiagram{2,1} 
 	&-& \ydiagram{2,2} 
 	&-& \ydiagram{3,2} 
 	&-& \ydiagram{1,3,2} \\
 	\downarrow \pi_1 &&&&& \downarrow \pi_1
 	\\
 	\begin{ytableau}
 	2 \\
 	*(gray) & *(gray) & 1 \\
 	*(gray) & 3
 	\end{ytableau}
 	& \ydiagram{2,1} 
 	&-& \ydiagram{2,2} 
 	&-& \ydiagram{1,2,2} 
 	&-& \ydiagram{1,3,2} \\	
 	\end{array}
 	\end{align*}
 \end{exa}
Let $\ab$ be a skew composition, $E\in \mc E(\ab)$ and $T_1,T_2 \in E$ be such that $T_1\pleq T_2$.
Recall that for each saturated chain from $T_1$ to $T_2$ in $E$ the index set of the 0-Hecke operators establishing the covering relations within the chain is  $\supp(\col_{T_2}\col^{-1}_{T_1})$.
As a consequence of \autoref{thm:dominance_and_hecke_covering} we obtain a criterion for determining whether an operator $\pi_i$ appears in the saturated chains from $T_1$ to $T_2$ or not.

\begin{prop}
	\label{thm:support_of_column_word_and_shape}
	Let $\ab$ be a skew composition of size $n$,  $i\in [n-1], E\in \mc E(\ab)$ and $T_1,T_2\in E$ be such that $T_1\pleq T_2$. 
	The following statements are equivalent.
	\begin{compactenum}
		\item $i\in \supp(\col_{T_2}\col^{-1}_{T_1})$.
		\item $\sh(T_2^{>i}) \neq \sh(T_1^{>i})$.
	\end{compactenum}
\end{prop}

\begin{proof}	 
	\autoref{thm:dominance_and_hecke_covering} applied to each covering relation in a saturated chain from $T_1$ to $T_2$ in $E$ and the fact that $\dleq$ is a preorder imply 
	\begin{align*}
		i\in \supp(\col_{T_2}\col^{-1}_{T_1}) \iff \osh(T_2^{>i}) \neq \osh(T_1^{>i}).
	\end{align*}
	From this the claim follows since $\ish(T^{>i}_1) = \beta = \ish(T^{>i}_2)$.
\end{proof}

\section{The endomorphism ring of \texorpdfstring{$\bs S_{\alpha,E}$}{S\_a,E}}
\label{sec:endomorphism_ring_of_0-Hecke_modules}

 For each $\alpha \vDash n$ there is an equivalence class $E_\alpha\in \mc E(\alpha)$ such that for all $T \in E_\alpha$ the entries increase in each column from top to bottom \cite[Section 8]{Tewari2015}. 
In \cite{Tewari2015}, Tewari and van Willigenburg showed that $\bs S_{\alpha,E_\alpha}$ is indecomposable.

In this section, we show for all $E\in \mc E(\alpha)$ that  $\End_{\He}(\se) = \C\id$ and hence $\se$ is indecomposable; this extends the result of Tewari and van Willigenburg to the general case.
By \autoref{thm:decomposition_in_equivalence_classes} we then have the desired decomposition of $\sa$.
In contrast, skew modules $\sse$ can be decomposable. We give an example at the end of the section.

We fix some notation that we use in the entire section unless stated otherwise.
Let $\alpha\vDash n$, $E\in \mc E(\alpha)$ and $T_0 := T_{0,E}$ be the source tableau of $E$.
Moreover, let $f\in \End_{\He}(\se)$, $v:=f(T_0)$ and $v = \sum_{T\in E} a_T T$ be the expansion of $v$ in the $\C$-basis $E$.
Since $\se$ is cyclically generated by $T_0$, $f$ is already determined by $v$.
The \emph{support} of $v$ is given by $\supp(v)=\set{T\in E \mid a_T \neq 0}$.
Our goal is to show that $T_0$ is the only tableau that may occur in $\supp(v)$ since then $f = a_{T_0} \id \in \C\id$.
We begin with a property holding for $\supp(v)$ that also appeared in the proof of \cite[Theorem 7.8]{Tewari2015}.

\begin{lem}
	\label{thm:descent_set_of_support_of_v}
	If $T\in \supp(v)$ then $\D(T) \subseteq \D(T_0)$.
\end{lem}
\begin{proof} Let $T_*\in E$ be such that $\D(T_*)\not \subseteq \D(T_0)$. 
	Then there is an $i\in \D(T_*)\cap \DC(T_0)$. 
	 Because $i\in \DC(T_0)$, $\pi_i v = f(\pi_i T_0) = v$. Thus, $a_{T_*}$ is the coefficient of $T_*$ in $\pi_i v =  \sum_{T \in E} a_T \pi_i T$. But this coefficient is $0$ as  $\pi_iT \neq T_*$ for all $T\in E$. To see this assume that there is a $T\in E$ such that $\pi_i T = T_*$. Then we obtain a contradiction as
	 \[
	 T_* \neq \pi_i T_* = \pi_i^2 T = \pi_i T = T_*.  \qedhere
	 \]	
\end{proof}

Let $T\in E$ be such that $T\neq T_0$ and $D(T)\subseteq D(T_0)$.
Thanks to  \autoref{thm:descent_set_of_support_of_v} it remains to show $a_T=0$.  To do this we use a 0-Hecke operator $\pi_\sigma$ where $\sigma= s_{j-1}\cdots s_i$ and $i$ and $j$ are given by
\begin{align}
	 \label{eq:def_of_i_and_j}
	 \begin{aligned}
	 i &= \max \big\{ k \in [n] \mid T^{-1}(k) \neq T^{-1}_0(k) \big\}, \\
	 j &= \min \big\{k \in [n] \mid k>i \text{ and } i\att_{T_0} k \big\}.
	 \end{aligned}
\end{align}
That is, $i$ is the greatest entry whose position in $T$ differs from that in $T_0$ and $j$ is the smallest entry in $T_0$ which is greater than $i$ and attacked by $i$ in $T_0$. At this point it is not clear that $j$ is well defined, and the following two lemmas are devoted to show this. 

\begin{samepage}
\begin{exa}
	Consider the equivalence class $E$ from \autoref{fig:example_for_E}. Then $T_0 = T_{0,E}$ and there are exactly one other tableau $T$ in $E$ with $\D(T)\subseteq \D(T_0)$:
	\begin{align*}
	T_0
	= 
	\begin{ytableau} 
	1 \\ 
	6 & 5 & \bs 4 & 3 \\ 
	8 & 7 & \bs 2 \\ 
	\end{ytableau}  
	\overset{\pi_1}{\longrightarrow}
	T = 
	\begin{ytableau} 
	\bs 2 \\ 
	6 & 5 & \bs 4 & 3 \\ 
	8 & 7 & 1 \\ 
	\end{ytableau}
	\end{align*}
	Defining $i$ and $j$ for $T$ as in \eqref{eq:def_of_i_and_j}, we obtain $i=2$ and $j=4$.
	Note that $2\in D(T_0)$. This property holds in general by the following result.
\end{exa}	
\end{samepage}

\begin{lem} 	
	\label{thm:i_is_descent}
	Let $T\in E$ be such that $T\neq T_0$ and $\D(T)\subseteq \D(T_0)$ and set
	\begin{align*}
	i = \max \big\{k \in [n] \mid T^{-1}(k) \neq T^{-1}_0(k)\big\}.
	\end{align*}
	Then $i\in \D(T_0)$.
\end{lem}

\begin{proof}
	Let $T$, $T_0$ and $i$ be given as in the assertion. 
	We introduce indices such that
	$D(T_0)= \set{ d_1 < d_2 < \cdots < d_m}$ and set $d_0 := 0, d_{m+1} := n$.
	 Define $I_k$ := $[d_{k-1}+1 , d_k]$ for $k= 1, \dots, m+1$. 
	 Recall that since $T_0$ is a source tableau, $\DC(T_0)= \ND(T_0)$ by \autoref{thm:source_and_sink_tableau}. 
	 That is, $a+1$ is the left neighbor of $a$ for each ascent $a$ of $T_0$. Therefore, we have $I_k\setminus \set{d_k} \subset \ND(T_0)$ and conclude that $T^{-1}_0(I_k)$ is a connected horizontal strip (a one-row diagram which contains all cells between its leftmost and rightmost cell) for $k = 1,\dots,m+1$.
	 
	Set $\cell_k := T_0^{-1}(k)$ for $k=1,\dots, n$ and let $x$ be the index such that $T(\cell_x)=~i$. 
	Since $T_0$ and $T$ are straight tableaux, the ordering conditions of SCTx imply $T^{-1}(n)= (l(\alpha), 1) = T_0^{-1}(n)$. Therefore  $i\neq n$ and we now show $i\notin \DC(T_0)$.	
	
	Assume for sake of contradiction that $i\in \DC(T_0)$.
	Let $l\in [m+1]$ be such that $i\in I_l$. 
	Since $i \in \DC(T_0)$, $i<d_l$ and $i+1\in I_l$. 
	The strip $T^{-1}_0(I_l)$ looks as follows:
	\begin{align}
	\label{eq:horizontal_strip_of_i}
	\cell_{d_{l}} \cell_{d_{l} -1} \cdots  \cell_{i+1} \cell_i \cdots \cell_{d_{l-1} + 1}
	\end{align}	
%
%
%
	By choice of $i$ we have 
	\begin{align}
	\label{eq:entries_of_T_greater_i}
	\text{$T(\cell_k) = k$ for $k=i+1,\dots, n$ and 	$T(\cell_i) < i$.}
	\end{align}
	Since entries decrease in rows of $T$, \eqref{eq:horizontal_strip_of_i} implies 
	\begin{align}
	\label{eq:entries_of_T_in_horizontal_strip_smaller_i}
	T(\cell_k) < i \text{ for } k=d_{l-1}+1,\dots , i.  
	\end{align}
	Combining  \eqref{eq:entries_of_T_greater_i} and \eqref{eq:entries_of_T_in_horizontal_strip_smaller_i} we obtain
	\begin{align}
	\label{eq:bound_for_x}
	 x\leq d_{l-1}.
	\end{align}
	We deal with two cases depending on $c_T(i)$. 
	In both cases we will end up with a contradiction.
	
	\textbf{Case 1} $c_T(i) \leq c_{T_0}(d_{l-1}+1)$. From $\D(T) \subseteq \D(T_0)$ follows $i\in \DC(T)$ and thus $c_T(i+1) < c_T(i) $. 
	Using $c_{T_0}(i) = c_{T_0}(i+1) + 1 =  c_{T}(i+1) + 1$, we obtain
	$c_{T_0}(i) \leq c_T(i) \leq  c_{T_0}(d_{l-1}+1).$
	Then there is a $y\in [d_{l-1}+1, i]$ such that $\cell_x$ and $\cell_y$ are in the same column. 
	On the one hand, from \eqref{eq:entries_of_T_in_horizontal_strip_smaller_i} follows $T(\cell_y) < i = T(\cell_x)$. 
	On the other hand, the choice of $y$ and \eqref{eq:bound_for_x} imply $y>d_{l-1} \geq x$ and hence $T_0(\cell_y) =y> x = T_0(\cell_x)$.	
	That is, in the column of $\cell_x$ and $\cell_y$ the relative order of entries in $T$ differs from that in $T_0$.
	So $T\not \sim T_0$ which contradicts the assumption $T,T_0 \in E$.
	
	\textbf{Case 2}
	$c_T(i) > c_{T_0}(d_{l-1}+1)$. This case is illustrated in \autoref{fig:i_is_descent}. Since by \eqref{eq:bound_for_x} $x\leq d_{l-1}$, there is a $1\leq p\leq l-1$ such that $x\in I_p$.
	The leftmost cell of the connected horizontal strip $T^{-1}_0(I_p)$ is $\cell_{d_p}$. As entries decrease in rows of $T$ from left to right, we have $T(\cell_{d_p})\geq T(\cell_x) = i$. In addition, from the choice of $p$ and \eqref{eq:entries_of_T_greater_i} follows that $T(\cell_{d_p})\leq i$. Thus, $d_p=x$.
	
	From $d_p=x$ we obtain $d_p \neq d_{l-1}$ since
	\begin{align*}
	c_{T_0}(d_{l-1}) \leq c_{T_0}(d_{l-1}+1) < c_T(i) = c_{T_0}(d_p)
	\end{align*}
	where the first inequality follows from $d_{l-1}\in \D(T_0)$.
	We claim that there exists an index $y\in [d_p +1, d_{l-1} -1]$ such that $\cell_y$ and $\cell_{d_p}$ are located in the same column.
	To prove the claim, assume for sake of contradiction that this is not the case. Then $d_p \in \D(T_0)$ implies $ c_{T_0}(d_p)< c_{T_0}(d_p+1)$. Thus, from $\DC(T_0)= \ND(T_0)$ and induction follows $c_{T_0}(d_p)< c_{T_0}(z)$ for all $z\in [d_p +1, d_{l-1} -1]$. As a consequence
		\begin{align*}
		c_{T_0}(d_{l-1}) < c_{T_0}(d_p) < c_{T_0}(d_{l-1}-1).
		\end{align*}
		In other words,  $d_{l-1}-1$ is an ascent of $T_0$ which is not a neighbor of $d_l$. This is a contradiction to the fact that $T_0$ is a source tableau and finishes the proof of the claim.
	
	Now, let $y$ be as claimed above.  Then $y\in  [d_p +1, d_{l-1} -1]$ and in particular  $y\neq d_p= x$. Hence, \eqref{eq:entries_of_T_greater_i} implies $T(\cell_y) <i$ and so $T(\cell_y) < i = T(\cell_{d_p})$ . 
 On the other hand, $y\in  [d_p +1, d_{l-1} -1]$ yields $T_0(\cell_y) =y > d_p = T_0(\cell_{d_p})$. As in Case 1, this is a contradiction to $T,T_0\in E$.
\end{proof}

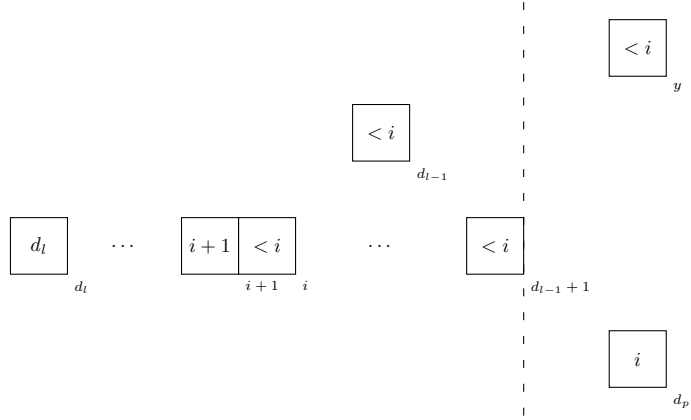
\begin{figure}
	\begin{tikzpicture}[scale=.75, transform shape]
	
	\def \x{-3}
	\def\y{0}
	\node at (\x,\y) {$d_l$};
	\draw  (\x-0.5,\y+0.5) rectangle (\x+0.5,\y-0.5);
	\node[below right]  at (\x+0.5,\y-0.5) {\scriptsize $d_l$};
	
	\def \x{-1.5}
	\def\y{0}
	\node at (\x,\y) {$\dots$};
	
	\def \x{0}
	\def\y{0}
	\node at (\x,\y) {$i+1$};
	\draw  (\x-0.5,\y+0.5) rectangle (\x+0.5,\y-0.5);
	\node[below right]  at (\x+0.5,\y-0.5) {\scriptsize $i+1$};
	
	\def \x{1}
	\def\y{0}
	\node at (\x,\y) {$<i$};
	\draw  (\x-0.5,\y+0.5) rectangle (\x+0.5,\y-0.5);
	\node[below right]  at (\x+0.5,\y-0.5) {\scriptsize $i$};
	
	\def \x{3}
	\def\y{0}
	\node at (\x,\y) {$\dots$};
	
	\def \x{3}
	\def\y{2}
	\node at (\x,\y) {$<i$};
	\draw  (\x-0.5,\y+0.5) rectangle (\x+0.5,\y-0.5);
	\node[below right]  at (\x+0.5,\y-0.5) {\scriptsize $d_{l-1}$};
	
	\def \x{5}
	\def\y{0}
	\node at (\x,\y) {$<i$};
	\draw  (\x-0.5,\y+0.5) rectangle (\x+0.5,\y-0.5);
	\node[below right]  at (\x+0.5,\y-0.5) {\scriptsize $d_{l-1}+1$};
	
	\draw[loosely dashed] (\x+0.5,\y-3) -- (\x+0.5,\y+4.5);
	
	\def \x{7.5}
	\def\y{-2}
	\node at (\x,\y) {$i$};
	\draw  (\x-0.5,\y+0.5) rectangle (\x+0.5,\y-0.5);
	\node[below right]  at (\x+0.5,\y-0.5) {\scriptsize $d_{p}$};
	
	\def\y{3.5}
	\node at (\x,\y) {$<i$};
	\draw  (\x-0.5,\y+0.5) rectangle (\x+0.5,\y-0.5);
	\node[below right]  at (\x+0.5,\y-0.5) {\scriptsize $y$};
	
	\end{tikzpicture} 
	\caption{
		\label{fig:i_is_descent}
		An example for the positions of cells and entries in the tableau $T$ from Case 2 of the proof of \autoref{thm:i_is_descent}.}
\end{figure}

Note that the $i$ appearing in the following Lemma is not the same as in \eqref{eq:def_of_i_and_j}.

\begin{lem}
	\label{thm:attacked_entry_exists}
	For all $i\in \D(T_0)$ there exists a $k\in T_0$ such that $k>i$ and $i \att_{T_0} k$.
\end{lem}
\begin{proof}
	Let $i\in \D(T_0)$. Then $c(i) \leq c(i+1)$ and thus $r(i) \neq r(i+1)$. Since $T_0$ is straight by assumption, the cell $(r(i+1), c(i))$ is contained in the shape of $T_0$. Let $k$ be the entry of $T_0$ in that cell. Then $i\att k$ and $k\geq i+1 $  as entries decrease in rows.
\end{proof}


Let $T$, $i$ and $j$ as in \eqref{eq:def_of_i_and_j}. From \autoref{thm:i_is_descent} and \autoref{thm:attacked_entry_exists} follows that $j$ is well defined. We proceed by considering the relative positions of $i$ and $[i+1,j]$ first in $T_0$ and then in $T$. This will allow to deduce useful properties of the operator $\pi_\sigma$ to be defined in \autoref{thm:end:annihilator_for_T_0}.  In the following Lemma, $i$ is slightly more general as in \eqref{eq:def_of_i_and_j}.

\begin{lem}
	\label{thm:position_of_entries_smaller_as_attacked_entry_in_T_0}
	Let $i\in \D(T_0)$ and set $j = \min \{k \in [n] \mid k>i \text{ and } i\att_{T_0} k \}$.
	Then $j$ is well defined and in $T_0$ $i$ is located strictly left of $[i+1,j-1]$ and does not attack $[i+1,j-1]$.
\end{lem}

We illustrate \autoref{thm:position_of_entries_smaller_as_attacked_entry_in_T_0} before we prove it.

	\begin{exa} For the source tableau from above
		\[
			T_0
			= 
\begin{ytableau} 
1 \\ 
6 & 5 & \bs4 & \bs3 \\ 
8 & 7 & \bs2 \\ 
\end{ytableau}  
		\]
		and $i=2\in \D(T_0)$ we have $j = 4 =\min \{k \in [n] \mid k>i \text{ and } i\att_{T_0} k \}$ and  $\set{3} = [i+1,j-1]$. Note $2\att 4$ but $2\not \att 3$.
	\end{exa}

	\begin{proof}[Proof of \autoref{thm:position_of_entries_smaller_as_attacked_entry_in_T_0}]
		From \autoref{thm:attacked_entry_exists} follows that $j$ is well defined. 
		We set $I =: [i+1,j-1]$ and $c_l := c_{T_0}(l)$ for $l \in T_0$. By the minimality of $j$ we have $i\not \att_{T_0} I$. It remains to show that $i$ is strictly left of $I$ or equivalently that $c_i < c_l$ for all $l\in I$. We may assume $I\neq \emptyset$ and prove this by induction on the elements of $I$. 
		 Since $i\in \D(T_0)$, $c_i \leq c_{i+1}$. Moreover  $i+1\in I$ implies $i\not \att_{T_0} i+1$ and consequently $c_i < c_{i+1}$.
		
			Now assume $l>i+1$ and $c_i < c_{l-1}$. If $l-1\in \D(T_0)$ then $c_i < c_{l-1} \leq c_l$. If $l-1\in \DC(T_0)$ then $l-1\in \ND(T_0)$ as $T_0$ is a source tableau. Thus $c_l= c_{l-1}-1$ and $c_i \leq c_l$. Furthermore $c_i \neq  c_l$ since $i\not \att_{T_0} I \ni l$. \qedhere
	\end{proof}

Let $T,i$ and $j$ as in \eqref{eq:def_of_i_and_j}. By definition $i\att j$ in $T_0$. In contrast, the next Lemma shows $i\natt j$ in $T$. There, $i$ and $j$ are defined as in \eqref{eq:def_of_i_and_j}.

\begin{lem}
	\label{thm:position_of_entries_smaller_as_attacked_entry_in_T}
	Let  $T\in E$ be such that $T\neq T_0$ and $\D(T) \subseteq \D(T_0)$. Define
	\begin{align*}
		i &= \max \big\{k \in [n] \mid T^{-1}(k) \neq T^{-1}_0(k)\big\}, \\
		j &= \min \big\{k \in [n] \mid k>i \text{ and } i\att_{T_0} k\big\}.
	\end{align*}
	Then $i$ and $j$ are well defined and in $T$  $i$ appears strictly left of $[i+1,j]$ and does not attack $[i+1,j]$. 
\end{lem}

We give an example before starting the proof of \autoref{thm:position_of_entries_smaller_as_attacked_entry_in_T}\!.

\begin{exa} Recall that in our running example $i=2$ and $j=4$ when defined for
	\begin{align*}
		T &= 
		\begin{ytableau} 
			\bs 2 \\ 
			6 & 5 & \bs4 & \bs3 \\ 
			8 & 7 & 1 \\ 
		\end{ytableau}
	\end{align*} 
	 as in \autoref{thm:position_of_entries_smaller_as_attacked_entry_in_T}. Then
   $[i+1,j] = \set{3,4}$ and $2\not \att \set{3,4}$.
\end{exa}

\begin{proof}[Proof of \autoref{thm:position_of_entries_smaller_as_attacked_entry_in_T}] 
	\autoref{thm:i_is_descent} yields $i\in \D(T_0)$ and so  \autoref{thm:attacked_entry_exists} ensures that $j$ is well defined. Set $\sigma := \col_T\col^{-1}_{T_0}$, $\cell_k := T_0^{-1}(k)$ for $k=1,\dots, n$ and let $x$ be the index such that $T(\cell_x)= i$. 

	 By choice of $i$, we have $T^{>i} = T_0^{>i}$. So, $\sh (T^{>k}) = \sh (T_0^{>k})$ for $k = i ,\dots, n$. Hence, from \autoref{thm:support_of_column_word_and_shape}  we obtain  
		\begin{align}
		\label{eq:support_of_path_to_T}
			\supp(\sigma) \subseteq [i-1].	
		\end{align}
	Let $\fs_{i_p} \cdots \fs_{i_1}$ be a reduced word for $\sigma$.
	Then $T = \pi_{i_p} \cdots \pi_{i_1} T_0$. From \eqref{eq:support_of_path_to_T} we have $i_q \neq i$ for $q=1,\dots, p$.
	Moreover, at least one $\pi_{i_q}$ has to move $i$ because the position of $i$ in $T$ differs from its position in $T_0$. Hence, there is a $q$ such that $i_q = i-1$ since $\pi_{i-1}$ and $\pi_i$ are the only operators that are able to move $i$.
	For two SCTx $T_1$ and $T_2$  such that $T_2 = \pi_{i-1}T_1= \fs_{i-1} T_1$ we have that $i-1\in \nAD (T_1)$ and thus $T_2^{-1}(i)$ is left of $T_1^{-1}(i)$ and $T_2^{-1}(i) \not \att T_1^{-1}(i)$. So, by applying $\pi_{i_p} \cdots \pi_{i_1}$ to $T_0$, $i$ is moved (possibly multiple times) strictly to the left into a cell that does not attack $\cell_i$. That is,
	\begin{align}
		\label{eq:postitoin_of_i_in_T_and_T_0}
		\text{$\cell_x$ is located strictly left of $\cell_i$ and $\cell_x\not \att \cell_i$.}
	\end{align}
	From the definition of $i$ follows that the entries $[i+1,j-1]$ have the same position in $T$ and $T_0$. By combining \eqref{eq:postitoin_of_i_in_T_and_T_0} and \autoref{thm:position_of_entries_smaller_as_attacked_entry_in_T_0} we obtain
	\begin{align}
	\text{In $T$ $i$ is located strictly left of $[i+1,j-1]$ and $i\not \att_T [i+1,j-1]$.}
	\end{align}
	Recall that $j$ has the same position in $T$ and $T_0$. 
	From \eqref{eq:postitoin_of_i_in_T_and_T_0} and $i\att_{T_0} j$ follows $c_T(i) < c_{T_0}(i)\leq c_{T_0}(j)$. Thus, $i$ is strictly left of $j$ in $T$.
	
	It remains to show $i\not \att_T j$.
	Since $i\att_{T_0} j$ either $c_{T_0}(j)=c_{T_0}(i)+1$ or $c_{T_0}(j)=c_{T_0}(i)$.
	
	 \textbf{Case 1} $c_{T_0}(j)=c_{T_0}(i)+1$. Then \eqref{eq:postitoin_of_i_in_T_and_T_0} implies $c_T(i) < c_{T_0}(i)< c_{T_0}(j) = c_T(j)$ and so $i\not \att_T j$. 
	 
	 \textbf{Case 2} $c_{T_0}(j)=c_{T_0}(i)$. If $c_T(i) < c_{T_0}(i)-1$ then $c_T(i) < c_T(j)-1$ and so $i\natt_T j$. If $c_T(i) = c_{T_0}(i)-1$ then $i$ and $j$ appear in adjacent columns of $T$ and for $i\natt_T j$ we have to show that $r_T(j) < r_T(i)$.
	  On the one hand, we have $1\leq c_T(i) <c_{T_0}(i)$ so that $i$ has a left neighbor $t>i$ in $T_0$.
	  In addition, from the first statement of  \autoref{thm:position_of_entries_smaller_as_attacked_entry_in_T_0} and  $c_{T_0}(j)=c_{T_0}(i)$ follows that $i$ is weakly left of $[i+1,j]$ in $T_0$.
	  Thus, $t>j$ and hence $r_{T_0}(j) < r_{T_0}(i)$ because otherwise $t,i$ and $j$ would violate the triple rule in $T_0$.
	 On the other hand,  $c_T(i) = c_{T_0}(i)-1$ and $i\not \att_T \cell_i$ imply $r_{T_0}(i) < r_T(i)$. All in all, $r_{T}(j) = r_{T_0}(j) < r_{T_0}(i) < r_T(i)$ and thus $i \not \att_T j$.
\end{proof}

Next, we prove useful properties of the operators $\pi_\sigma$ mentioned already in \eqref{eq:def_of_i_and_j}.

\begin{lem}
	\label{thm:end:annihilator_for_T_0}
	Keep the notation  of \autoref{thm:position_of_entries_smaller_as_attacked_entry_in_T} and set $\sigma = \fs_{j-1} \cdots \fs_{i+1} \fs_i$. Then
	\begin{compactenum}
		\item $\pi_\sigma T_0 = 0$,
		\item $\pi_\sigma T \in E$,
		\item $\sigma = \col_{\pi_\sigma T} \col^{-1}_{T}$.
	\end{compactenum}
\end{lem}

\begin{proof}
	First observe that $\fs_{j-1} \cdots \fs_{i+1} \fs_i$ is a reduced word, i.e., $\pi_\sigma = \pi_{j-1} \cdots \pi_{i+1}\pi_i$. Set $\cell_k = T_0^{-1}(k)$ for $k=1,\dots, n$.
	
	We consider $T_0$. Set $T' = \pi_{j-2} \cdots \pi_{i+1}\pi_i T_0$. 
	We can apply \autoref{thm:mutliple_flips_same_cell} in $T_0$ on $i$ and $[i+1,j-1]$ because of \autoref{thm:i_is_descent} and \autoref{thm:position_of_entries_smaller_as_attacked_entry_in_T_0}. Doing this we obtain that $T'\in E$ and $T'(\cell_i) = j-1$. In addition, $T'(\cell_j) = T_0(\cell_j) = j$ as none of the operators $\pi_{j-2}, \dots, \pi_{i+1}$ moves $j$. Recall that $j$ is defined such that $\cell_i \att \cell_j$. Thus $j-1\in \AD(T')$ and $\pi_\sigma T_0 = 0$.

	Now consider $T$. Because of  \autoref{thm:position_of_entries_smaller_as_attacked_entry_in_T} we can apply \autoref{thm:mutliple_flips_same_cell} in $T$ on $i$ and $[i+1,j]$. This immediately gives us $(2)$ and $(3)$.
\end{proof}

\begin{exa} Continuing our running example, we have $i = 2$, $j = 4$ and $\pi_\sigma=\pi_3\pi_2$. Moreover,
	\begin{align*}
		T_0 &=
		\begin{ytableau} 
			1 \\ 
			6 & 5 & 4 & 3 \\ 
			8 & 7 & 2 \\ 
		\end{ytableau}  	
		\overset{\pi_2}{\longrightarrow}
		\begin{ytableau} 
			1 \\ 
			6 & 5 & 4 & 2 \\ 
			8 & 7 & 3 \\ 
		\end{ytableau} 	
		\overset{\pi_3}{\longrightarrow}
		0,
		\\
		T &= 
		\begin{ytableau} 
			2 \\ 
			6 & 5 & 4 & 3 \\ 
			8 & 7 & 1 \\ 
		\end{ytableau}  
		\overset{\pi_2}{\longrightarrow}
		\begin{ytableau}
			3 \\ 
			6 & 5 & 4 & 2 \\ 
			8 & 7 & 1 \\ 
		\end{ytableau}
		\overset{\pi_3}{\longrightarrow}
		\begin{ytableau}
			4 \\ 
			6 & 5 & 3 & 2 \\ 
			8 & 7 & 1 \\ 
		\end{ytableau}.
	\end{align*}
\end{exa}

We are ready to prove the main result of this paper now.

\begin{thm}	\label{thm:end:end_is_c}
	Let $\alpha \vDash n$ and $E\in \mc E(\alpha)$. Then $\End_{\He}(\bs S_{\alpha,E})=\C \id$. In particular, $\se$ is an indecomposable $\He$-module.
\end{thm}

\begin{proof}  For the second part note that if  $\End_{\He}(\bs S_{\alpha,E})=\C \id$ then clearly $\se$ is indecomposable. 
	
	To prove the first part, let $f\in \End_{\He}(\se)$, $v:=f(T_0)$ and $v = \sum_{T\in E} a_T T$ as before.
	We show $\supp(v) \subseteq \{T_0\}$ since this and the fact that $\se$ is cyclically generated by $T_0$ imply $f = a_{T_0}\id\in \C \id$.
	
	If $v=0$ this is clear so that we can assume $v\neq 0$.
	Let $T_*\in \supp(v)$ be of maximal degree in $E$. Assume for sake of contradiction that $T_* \neq T_0$. Then \autoref{thm:descent_set_of_support_of_v} yields $\D(T_*) \subseteq  \D(T_0)$. Hence \autoref{thm:end:annihilator_for_T_0} provides the existence of $\sigma \in \fS_n$ such that $\pi_\sigma T_*\in E$, $\pi_\sigma T_0 = 0$ and $\sigma = \col_{\pi_\sigma T_*}\col_{T_*}^{-1}$.
	
	We claim that if $T\in \supp(v)$ and $\pi_\sigma T = \pi_\sigma T_*$ then $T = T_*$.
	To see this, let  $T\in \supp(v)$ be such that $\pi_\sigma T = \pi_\sigma T_*$. Then
		\begin{alignat*}{3}
			l(\sigma) 
			\geq \delta(\pi_\sigma T) - \delta(T)
			= \delta(\pi_\sigma T_*) - \delta(T) 
			\geq \delta(\pi_\sigma T_*) - \delta(T_*) 
			= l(\sigma)
		\end{alignat*}
		where \autoref{thm:rank_in_E_and_column_word} is used to establish the first inequality and the last equality. Hence, 	$l(\sigma)= \delta(\pi_\sigma T) - \delta(T)$ and another application of \autoref{thm:rank_in_E_and_column_word} yields $\col_{\pi_\sigma T_*}\col^{-1}_{T} = \sigma$. But then $\col_T = \col_{T_*}$ so that $T = T_*$ as claimed.
		
	The claim implies that the coefficient of $\pi_\sigma T_*$ in $\pi_\sigma v = \sum_{T \in \supp(v)} a_T \pi_\sigma T$ is $a_{T_*}$.	
	Yet, $\pi_\sigma v = f(\pi_\sigma T_0) = 0$ and hence $a_{T_*} = 0$ which contradicts the assumption $T_* \in \supp(v)$ and completes the proof of $\supp(v) \subseteq \set{T_0}$.
\end{proof}

Combining  \autoref{thm:end:end_is_c} with \autoref{thm:decomposition_in_equivalence_classes},  we obtain the desired decomposition of $\sa$.

\begin{cor}
	Let $\alpha \vDash n$. Then $\sa = \bigoplus_{E\in \mc E(\alpha)} \se$ is a decomposition into indecomposable submodules.
\end{cor}

\begin{exa}
	\label{exa:decomposable_module}
	 In general, \autoref{thm:end:end_is_c} does not hold for skew modules $\sse$.
	The two tableaux
	\begin{align*}
	T_0 =
	\begin{ytableau}
	1 \\
	*(gray) & *(gray) & 2
	\end{ytableau}
	\overset{\pi_1}{\longrightarrow}
	T_1 = 
	\begin{ytableau}
	2 \\
	*(gray) & *(gray) & 1
	\end{ytableau}
	\end{align*}
	form an equivalence class $E$. Let $n=2$ and $\ab=\sh(T_0)$. 
	Observe that we obtain an idempotent $\He$-endomorphism $\varphi$  by setting $\varphi(T_0) = \varphi(T_1) = T_1$.
	 Clearly, $\varphi$ is none of the trivial idempotents $0,\id\in \End_{\He}(\sse)$. Thus, $\End_{\He}(\sse)\neq \C\id$. Moreover, we obtain a decomposition
	\begin{align*}
	\sse = \varphi(\sse) \oplus (1-\varphi) (\sse) = \spa_\C(T_1) \oplus \spa_\C(T_0-T_1)
	\end{align*} 
	in two submodules of dimension $1$.
	
	This example also illustrates how the argumentation of this section can fail when it is applied to skew modules. Note that $D(T_1) \subseteq D(T_0)$. So, we may try to set
	 \begin{align*}
		i &= \max \big\{ k \in [n] \mid T^{-1}(k) \neq T^{-1}_0(k) \big\}, \\
		j &= \min \big\{k \in [n] \mid k>i \text{ and } i\att_{T_0} k \big\}.
	\end{align*}
	as before.
	But then $i=2$ so that $j$ does not exist. 
\end{exa}

\begin{ack} 
	The new results presented in this article are part of my PhD research, and I would like to thank my supervisor Christine Bessenrodt for her support especially during the work on this paper.
\end{ack}


\providecommand{\bysame}{\leavevmode\hbox to3em{\hrulefill}\thinspace}

\end{document}